\documentclass{amsart}
\usepackage{graphicx} 

\usepackage[dvipsnames]{xcolor}
\usepackage[draft]{say}
\definecolor{MHcol}{RGB}{100, 180, 255}

\definecolor{egyptianblue}{rgb}{0.06, 0.2, 0.65}

\usepackage{amsthm}
\usepackage{amssymb}
\usepackage{amsmath}
\usepackage{parskip}
\usepackage{tikz}
\usepackage{tikz-cd}
\usepackage{bbm}
\usepackage{enumitem}
\usepackage{xcolor}
\usepackage{parskip}
\usepackage{import}
\usepackage{yhmath}
\usepackage{aliascnt}
\usepackage{standalone}
\usepackage{caption}
\usepackage{subcaption}
\usepackage{mathtools}

\usepackage{comment}

\usepackage[margin=1in,marginparwidth=0.8in, marginparsep=0.1in]{geometry}

\usepackage[bookmarks=true, bookmarksopen=true,%
bookmarksdepth=3,bookmarksopenlevel=2,%
colorlinks=true,%
linkcolor=blue,%
citecolor=blue,%
filecolor=blue,%
menucolor=blue,%
urlcolor=blue]{hyperref}

\usepackage{mathtools}

\newcommand{\bC}{\mathbb{C}}

\renewcommand{\H}{\mathbb{H}}

\DeclareMathOperator{\Id}{Id}

\newcommand{\bR}{\mathbb{R}}

\newcommand{\cW}{\mathcal{W}}

\let\Im\relax
\DeclareMathOperator{\Im}{Im}

\newcommand{\R}{\mathbb R}
\newcommand{\C}{\mathbb C}
\newcommand{\id}{\mathrm{id}}
\newcommand{\abs}[1]{\lvert #1\rvert}

\newcommand{\norm}[1]{\lVert #1\rVert}

\DeclareMathOperator{\Fuk}{Fuk}

\newtheorem{dummy}{dummy}[section]
\newtheorem{lemma}[dummy]{Lemma}
\newtheorem{theorem}[dummy]{Theorem}

\newtheorem{corollary}[dummy]{Corollary}

\theoremstyle{definition}

\newtheorem{example}[dummy]{Example}
\newtheorem{remark}[dummy]{Remark}

\newtheorem{eg/lem}[dummy]{Example/Lemma}

\title{Symplectomorphism invariance of the wrapped Fukaya category}

\author{Mingyuan Hu} \author {Vivek Shende}

\begin{document}

\begin{abstract}      An exact symplectomorphism of finite type Liouville manifolds 
    is called `strict' if it has bounded primitive; such symplectomorphisms are well known to induce deformation equivalences and hence equivalences of wrapped Fukaya categories.  Here we note that an arbitrary symplectomorphism of Weinstein manifolds, with no assumptions at infinity, can be used to construct an equivalence of Fukaya categories.  We combine this with a   pointwise criterion for constructing such symplectomorphisms to  improve  two statements in the mirror symmetry literature -- for very affine hypersurfaces and multiplicative hypertoric varieties -- where Fukaya categories were calculated for a Liouville form not known to be strictly exact symplectomorphic to the natural K\"ahler Liouville form on the space in question. 
    All results are obtained using ChatGPT Astra.
\end{abstract}

\maketitle

\thispagestyle{empty}

\section{Introduction}

A manifold $Y$ carrying a 1-form $\lambda$ is said to be {\em exact symplectic} if $d \lambda$ is symplectic.
If in addition
the vector field $Z$ satisfying $d\lambda(Z, \cdot) = \lambda$ can be integrated for all time,  and there is a proper, bounded below function $\eta: Y \to \mathbb{R}$ such that $d\eta(Z) > 0$ outside a compact subset, one says $(Y, \lambda)$ is a finite type Liouville manifold.  (We will henceforth drop the words `finite type'.) 


Given exact symplectic manifolds $(Y, \lambda)$ and $(Y', \lambda')$, a map $f: Y \to Y'$ is said to be an {\em exact symplectomorphism} if $f^* \lambda - \lambda' = dF$, for some function $F$.  One says an exact symplectomorphism of Liouville manifolds is {\em strict} if $F$ may be chosen to vanish outside a compact set.

 A 1-form $\lambda$ on $Y \times [0,1]$ is said to be a Liouville deformation if one can choose a proper, bounded below, function $\eta: Y \times [0,1] \to \mathbb{R}$ so that $(Y, \lambda_t, \eta_t)$ satisfies the above properties and there is a constant \(c\) such that \( d\eta_t(Z_t) > 0\) on \(\{\eta_t \le c\} \) for all \(t \in [0, 1]\).

Let us recall:    any symplectomorphism of  Liouville manifolds is isotopic to an exact symplectomorphism \cite[Lemma 11.2]{CE} and  Liouville manifolds are deformation equivalent if and only if they are strictly exact symplectomorphic~\cite[Proposition 11.8]{CE}.  However, there are Liouville manifolds that are  exact symplectomorphic but not strictly exact symplectomorphic, as witnessed by their having non-diffeomorphic ideal contact boundaries \cite[Cor. 4.5]{Courte}. 

Nevertheless, it is known that the symplectic homology of Liouville manifolds is invariant under (not necessarily strict) exact symplectomorphism \cite{Seidel-biased, mclean-lefschetz}.
The purpose of this note is to record that the same is true for the wrapped Fukaya category, and to give general a criterion guaranteeing the existence of (not necessarily strict) exact symplectomorphisms.  

Wrapped Fukaya categories of Liouville manifolds were introduced in \cite{FSS, AbouzaidSeidel}; we follow the  definition in \cite{GPS1, GPS2}, and write $\Fuk(Y, \lambda)$ for this category.  We recall that the core of a Liouville manifold is the locus which does not escape under the Liouville flow, and a Liouville manifold is said to be {\em weakly Weinstein} if, for every component of the smooth locus of its core, there is a properly embedded Lagrangian disk which meets it once and transversely.  (Weinstein manifolds are  those for which the Liouville flow is gradient-like for a Morse function; in this case, after perturbing to Morse-Smale, said disks are  obtained from the cocores of the  critical points of index $\dim Y /2$.) Many results of \cite{GPS2}, relevantly here the existence of the Viterbo functor, are only available in the weakly Weinstein case.   Here we show:


\begin{theorem}\label{thm:W_of_different_Liouville}
    Let $(Y, \lambda)$ and $(Y', \lambda')$ be weakly Weinstein manifolds.  If  $(Y, d\lambda)$ and $(Y', d\lambda')$ are symplectomorphic, then there is an equivalence of categories
    $
    \Fuk(Y, \lambda) \simeq \Fuk(Y', \lambda')$.
\end{theorem}

We are motivated by the problem of computing wrapped Fukaya categories of smooth affine algebraic varieties, and the approach via microsheaves on Lagrangian skeleta.  It is often simpler to compute the skeleton for some hand-crafted Liouville form $\lambda$, rather than the natural primitive  $\lambda_K$ of a K\"ahler form; we would like to be able to conclude that such choices are irrelevant to the determination of the Fukaya category. 
To acheive this, we combine Theorem \ref{thm:W_of_different_Liouville} with a pointwise criterion for when two Liouville forms are symplectomorphic. 

Let \( (X, \omega_0 = d\beta) \) be a Liouville manifold. Fix an \(\omega_0\)-compatible almost complex structure \(J\), i.e.
\[
\omega_0(Jv, Jw) = \omega_0(v, w), \qquad \text{ and } g(v,w) \coloneqq \omega_0 (v, Jw)
\]
is a Riemannian metric. Let \(Z_\beta\) be the Liouville vector field, given by \(\iota_{Z_\beta} d\beta = \beta\). 

We also assume that there is a proper function \(R: X\to [1, \infty) \), such that outside a compact set of \(X\), \( \beta = - d R\circ J\) and \(Z_{\beta} (R) \le R\).  (Note this  implies  $dR(Z_\beta) >0$.)


\begin{example}
If  \( X = \overline{X} \cup_{\partial{\overline{X}}} (\partial \overline{X} \times [r, \infty) ) \), then we can choose \(R\) to be the projection to \( [r, \infty)\), and \( J \) sending \( Z_\beta\) to the Reeb field on the contact manifold \((\partial \overline{X} \times \{R\}, \frac{1}{R} \beta) \). In this case \(Z_\beta(R) = R\). 
\end{example}

\begin{lemma}\label{lem:stein}
Let \( X \subset \bC^n\) be a closed complex submanifold for the standard complex structure $I$.  We write  $z_\nu=x_\nu+i y_\nu$ for the holomorphic coordinates on $\bC^n$, and set: 
\[
 R=1+\frac14\sum_{\nu=1}^N|z_\nu|^2,
 \qquad
 \beta=-dR\circ I
      =\frac12\sum_{\nu=1}^N
        (x_\nu\,dy_\nu-y_\nu\,dx_\nu), \qquad g=d\beta(-,I-)
\]
Note $g$ is the Euclidean  metric  on $\mathbb C^N$.
Then \(R\) is proper on \(X\),  $Z_\beta(R) = |dR|_g^2 \le R - 1 < R$. 
In particular \(Z_\beta\) is complete.
\end{lemma}
\begin{proof}
\(R\) is proper on $X$ since it is proper on $\C^n$ and $X$ is a closed submanifold.  Moreover, 
 \(Z_\beta=\nabla_g(R|_X)\) is a projection of \(\nabla_g R=\frac12\sum_i(x_i\partial_{x_i}+y_i\partial_{y_i})\). Hence \(Z_\beta(R)=|dR|_X|_g^2\le|dR|_g^2 =\frac14|z|^2=R-1\). 
\end{proof}

For a symplectic form \(\omega\), we denote: 
\begin{equation*}\label{eq:inversenorm}
\omega^\flat(v)=\iota_v\omega,\qquad
\norm{\omega^{-1}}_g
\coloneqq\norm{(\omega^\flat)^{-1}}_{g^*, g}
\coloneqq \frac{1}{\displaystyle\inf_{\abs{v}_g=1}\abs{\iota_v\omega}_{g^*}}.
\end{equation*}

\begin{theorem}\label{thm:bound}
Fix $X, \omega_0 = d \beta, J, R$ as above.  
Let $\alpha$ be another  1-form on $X$, so that $\omega=d\alpha$ is a symplectic form {\emph weakly}  tamed by \(J\):
\begin{equation}\label{eq:weak_tame}
\omega(v,Jv)\ge0\qquad\text{for every }v\in TX.
\end{equation}
Suppose there is a constant $C>0$ and a smooth function $B:[1,\infty)\to(0,\infty)$ such that 
\begin{equation}\label{eq:growth_B}
B(r)\ge r,\qquad
\int_1^\infty\frac{r}{B(r)^2}\,dr=\infty,
\qquad \norm{\omega^{-1}}_g\le C B(R)
\end{equation}
outside a compact subset of $X$.
Then there is a diffeomorphism $F:X\to X$, smoothly isotopic to the identity, and a smooth function $h:X\to\R$ 
such that
\begin{equation}\label{eq:conclusion}
F^*\alpha=\beta+dh.
\end{equation}
\end{theorem}
Note that all the conditions are pointwise. In practice, we often take \( B(r) = r \) or \(B(r) = r \sqrt{1+\log (r)}\). 

\begin{corollary} \label{cor: equivalent}
    In the setup of Theorem \ref{thm:bound}, if \(\alpha\) and \(\beta\) are two (weakly) Weinstein structures on \(X\), then 
    \(\Fuk(X, \alpha) \cong \Fuk(X, \beta)\). 
\end{corollary}

We give the proof of Theorem~\ref{thm:bound} in Section \ref{sec: proof of thm:bound}.  
The proof is also formalized in \emph{Lean}; the formalization (using libraries from ~\cite{fltz}) is available at \url{https://github.com/Hu-Mingyuan/exact-liouville-comparison}. 

We have found the following special case to be repeatedly useful:

\begin{corollary} \label{cor: quadratic poisson}
Let $X \subset \C^N$ be a smooth affine algebraic variety; let $\beta$ be the standard Liouville form as in Lemma \ref{lem:stein}.
Let $\Omega$ be a holomorphic symplectic form on $X$, whose Poisson brackets are given by quadratic polynomials in the standard holomorphic coordinates $z_1, \ldots, z_N$.  

If $\alpha$ is a smooth real one-form satisfying $d\alpha=\operatorname{Im}\Omega$, then there is a diffeomorphism $\Phi:X\to X$, smoothly isotopic to
the identity, and a smooth function $f:X\to\mathbb R$ such that
$\Phi^*\beta=\alpha+df$.
\end{corollary}

We give the proof of  Corollary \ref{cor: quadratic poisson} in Section \ref{sec: quadratic poisson}.

\subsection*{AI Statement}
Special cases of all results stated above were discovered autonomously by ChatGPT Astra in the process of our investigations of mirror symmetry, the results of which will appear in future articles.  The main human contributions to the article were to seek formulations of the results in the natural generality, clarify the functoriality of Viterbo restriction,  select the illustrations in Section \ref{sec:applications} (the proofs of which were again provided by Astra), and provide the exposition.

\subsection*{Acknowledgements}
We are supported by Villum Fonden Villum Investigator 37814.

\section{Functoriality of Viterbo restriction}

We recall the construction of Viterbo restriction from \cite{GPS2}.
We write $\H$ for the half-plane sector. 
Recall that if $W$ is a Liouville sector, $X$ is a Liouville domain, and $X \subset \partial_\infty W$ is a Liouville hypersurface embedding, then there is a canonical (up to contractible choices) sector inclusion  $(\mathbb{H} \times X) \setminus Nbd(X) \subset   W  \setminus Nbd(X)$.  Here, we include $X \subset \partial_\infty (\H \times X)$ as the fiber over some point in the ideal boundary of $\H$.

Then for $Y \subset X$ a Liouville subdomain, there are maps $$\Fuk(X) \xhookrightarrow{s_X} \Fuk((\H \times X) \setminus Nbd(X)) \xrightarrow{v_{XY}} \Fuk((\H \times X) \setminus Nbd(Y))  \xleftarrow{i_Y} \Fuk((\H \times Y) \setminus Nbd(Y) ) \xhookleftarrow{s_Y} \Fuk(Y)$$
Here, $s_X$ and $s_Y$ multiply objects by the cotangent fiber, and $v_{XY}$ and $i_Y$ are all induced from sector inclusions.  

When $X, Y$ are weakly Weinstein, $s_X, s_Y$ and $i_Y$ are all equivalences; inverting $i_Y, s_Y$ defines the Viterbo restriction functor.

Given inclusions $Z \subset Y \subset X$, it is natural to expect the Viterbo restrictions to be functorial.  This was not explicitly checked in \cite{GPS2}.  One could deduce it by passing by \cite{GPS3} to microsheaves and invoke the functoriality of Viterbo restriction established in that context in \cite{Li-Nadler-Shende}.  Here we observe that functoriality of Viterbo restriction is immediate  from the following diagram: 

\begin{equation} \label{functoriality of Viterbo}
\begin{tikzcd}
	{\mathrm{Fuk}((\mathbb{H} \times X) \setminus Nbd(X))} & {\mathrm{Fuk}((\mathbb{H} \times X) \setminus Nbd(Y))} & {\mathrm{Fuk}((\mathbb{H} \times X) \setminus Nbd(Z))} \\
	& {\mathrm{Fuk}((\mathbb{H} \times Y) \setminus Nbd(Y))} & {\mathrm{Fuk}((\mathbb{H} \times Y) \setminus Nbd(Z))} \\
	&& {\mathrm{Fuk}((\mathbb{H} \times Z) \setminus Nbd(Z))}
	\arrow["{v_{XY}}", from=1-1, to=1-2]
	\arrow["{v_{XZ}}", bend left, from=1-1, to=1-3]
	\arrow[from=1-2, to=1-3]
	\arrow["\cong", from=2-2, to=1-2]
	\arrow["{v_{YZ}}", from=2-2, to=2-3]
	\arrow[from=2-3, to=1-3]
	\arrow["\cong"'{pos=0.8}, shift right=5, bend right, dashed, from=3-3, to=1-3]
	\arrow["\cong", from=3-3, to=2-3]
\end{tikzcd}
\end{equation}

Here, the dashed arrow is dashed only for typographical reasons.  The arrows marked $\cong$ are those established to be isomorphisms in the original argument from \cite{GPS2}; of course it follows that the remaining vertical arrow is also an isomorphism, hence that the unlabeled horizontal arrow can be identified with $v_{YZ}$. 

\section{Proof of Theorem \ref{thm:W_of_different_Liouville}}


Per \cite[Lemma 11.2]{CE}, we may find an exact symplectomorphism; pulling back by it, we reduce to the case of \( \alpha\) and \(\beta = \alpha + d f\)  two Liouville forms on \(X\).  
We choose submanifolds-with-boundary 
    \begin{equation}\
    \label{eq:four_domains}
     A_0\Subset B_0\Subset A_1\Subset B_1 \Subset X,
    \end{equation}
    such that \(A_i\) are Liouville domains for \(\alpha\) completing to $X$, and \(B_i\) are Liouville domains for \(\beta\) completing to $X$. We do not ask that the $\partial A_i$ is convex for $\beta$ or that $\partial B_i$ is convex for $\alpha$.

    Choose a  function \(\chi: X\to \bR\), such that 
\begin{equation}\label{eq:chi}
 \chi=0\text{ near }\partial A_0\cup\partial A_1,
 \qquad
 \chi=1\text{ near }\partial B_0\cup\partial B_1 \text{ and outside of} B_1, 
\end{equation}
and put
\( \lambda_*:=\alpha+d(\chi f) \). Then the four domains~\eqref{eq:four_domains} are Liouville subdomains of \( (M, \lambda_*)\). Hence we have the Viterbo restriction functors:
\[
\cW( B_1, \lambda_*) \xrightarrow{F_1} \cW(A_1, \lambda_*) \xrightarrow{F_2} \cW(B_0, \lambda_*) \xrightarrow{F_3} \cW(A_0, \lambda_*). 
\]
By condition~\eqref{eq:chi}, \( (B_i, \beta )\) and \((A_i, \alpha)\) are deformation equivalent to \( (B_i, \lambda_*) \) and \((A_i, \lambda_*)\), respectively. Hence,
\[
\cW(B_1, \lambda_*) \simeq \cW(B_0, \lambda_*) \simeq \cW(X, \beta), \qquad 
\cW(A_1, \lambda_*)  \simeq \cW(A_0, \lambda_*) \simeq \cW(X, \alpha).
\]
By the functoriality of Viterbo restriction \eqref{functoriality of Viterbo}, we have 
\[
F_2\circ F_1 \simeq \Id_{\cW(X, \beta)}, \qquad F_3 \circ F_2 \simeq \Id_{\cW(X, \alpha)}.
\]
Therefore, \(F_2\) induces an equivalence between \( \cW( X, \alpha)\) and \(\cW(X, \beta) \). 

\begin{remark}
The same argument implies symplectomorphism invariance of any 
invariant of Liouville manifolds which admits a Viterbo restriction functor.  
\end{remark}

\section{Proof of Theorem \ref{thm:bound}} \label{sec: proof of thm:bound}

The idea is to consider four paths of one forms:
\[
\alpha\longrightarrow\alpha+k\beta
\longrightarrow\alpha+A\beta
\longrightarrow A\beta
\longrightarrow\beta
\]
where \(k\) and \(A\) are some functions of \(R\) we choose later. Heuristically, \(k\) is small and \(A\) is large. For each path \(\mu_t\), use the Moser's trick and define the time-dependent vector field \(X_t\) by
\[
\iota_{X_t}d\mu_t  = -\dot \mu_t,
\]
and then solve the ODE:
\[
\frac d{dt}\phi_t(x)=X_t(\phi_t(x)),\qquad \phi_0 = \id. 
\]
Once we show that the solution exists for all \(t\in [0, 1]\),  by Cartan's formula,
\[
\frac d{dt}\phi_t^*\mu_t
=\phi_t^*\bigl(\dot\mu_t+\iota_{X_t}d\mu_t
+d(\mu_t(X_t))\bigr)
=d\bigl(\phi_t^*(\mu_t(X_t))\bigr),
\]
which implies
\begin{equation}\label{eq:moser_exact}
\phi_t^*\mu_t=\mu_0+dH_t,\qquad
H_t(x)=\int_0^t\mu_s(X_s)(\phi_s(x))\,ds.
\end{equation}
In particular, \(\phi_1\) is an exact symplectomorphism from \((X, d\mu_0)\) to \((X, d\mu_1)\). 
The only issue is therefore to show that each of the four Moser flows exists up to time \(1\). For this we use the following standard criterion.

\begin{lemma}\label{lem:complete}
Let \(X_t\), \(t\in[0,1]\), be a smooth time-dependent vector field on \(X\). Suppose that every solution of \(\dot x(t)=X_t(x(t))\) stays in a compact subset of \(X\) on its maximal interval of existence. Then for every \(s\in[0,1]\) and \(x\in X\), the solution with \(x(s)=x\) exists for all \(t\in[0,1]\). In particular, the maps \(\phi_t\) are defined on all of \(X\) and are diffeomorphisms.

The hypothesis holds if there is a proper function \(P:X\to\bR\), bounded from below, and a constant \(C\) such that \(|X_tP|\le C\) on \(X\) for all \(t\in[0,1]\). 
\end{lemma}

\begin{proof}
A solution of an ODE which stays in a compact set extends past any finite time, so all solutions exist on \([0,1]\). The inverse of \(\phi_t\) is obtained by solving the ODE backwards from time \(t\) to time \(0\), and is smooth by smooth dependence on initial conditions. For the last claim: along a solution, \(|P(x(t))-P(x(s))|\le C|t-s|\le C\), so \(x(t)\) stays in the compact set \(\{P\le P(x(s))+C\}\) for any \(s\).
\end{proof}

We will also repeatedly use the following observation. If \(d\mu_t\) is nondegenerate, evaluating the Moser equation \(\iota_{X_t}d\mu_t=-\dot\mu_t\) on the vector \(JX_t\) gives
\begin{equation}\label{eq:energy}
d\mu_t(X_t,JX_t)=-\dot\mu_t(JX_t).
\end{equation}
When \(d\mu_t\) tames \(J\), the left hand side is a positive definite quadratic form in \(X_t\), so~\eqref{eq:energy} bounds the size of \(X_t\) in terms of \(\dot\mu_t\).

\subsection{The first path}
Let us fix \(r_0>1\) such that~\eqref{eq:growth_B} and the identities \(\beta=-dR\circ J\), \(Z_\beta(R)\le R\) hold on \(\{R\ge r_0\}\). Note that \(\beta=-dR\circ J\) means \(g(Z_\beta,v)=d\beta(Z_\beta,Jv)=\beta(Jv)=dR(v)\), i.e. \(Z_\beta=\nabla_gR\), so that
\begin{equation}\label{eq:a}
|\beta|_g^2=|dR|_g^2=Z_\beta(R)\le R\qquad\text{on }\{R\ge r_0\}.
\end{equation}
Take a cutoff function
\[
\chi:[1,\infty)\longrightarrow[0,1],\qquad
\chi=0\ \text{ in }[1,r_0],\qquad
\chi=1\ \text{ in }[r_0+1,\infty)
\]
and define
\[
S(r)=1+\int_1^r\frac{u}{B(u)^2}\,du,
\qquad
k(r)=
\exp\left(
\int_1^r
\chi(u)\left(\frac{S'(u)}{2S(u)}-\frac1u\right)\,du
\right),
\qquad
m(r)=k(r)+rk'(r).
\]
We take
\[
\mu_t = \alpha + t k \beta, \qquad t\in[0,1],
\]
and show that the induced flow exists in \(t \in [0, 1]\).

\begin{lemma}\label{lem:k}
The function \(k\) satisfies \(k=1\) on \([1,r_0]\), \(0<k\le1\), \(k'\le0\), and \(m>0\). For \(r\ge r_0+1\) there is a constant \(b>0\) such that
\[
k(r)=b\frac{\sqrt{S(r)}}r,\qquad
m(r)=\frac{br}{2B(r)^2\sqrt{S(r)}}.
\]
In particular \(rk(r)\to\infty\) as \(r\to\infty\).
\end{lemma}
\begin{proof}
Since \(B(u)\ge u\) and \(S\ge1\), \(\frac{S'(u)}{2S(u)}=\frac{u}{2B(u)^2S(u)}\le\frac1{2u}\). So the integrand in the definition of \(k\) is \(\le0\) and vanishes on \([1,r_0]\). This gives the first three claims. Differentiating,
\[
m(r)=k(r)\left(
1-\chi(r)+\chi(r)\frac{r^2}{2B(r)^2S(r)}
\right)>0.
\]
For \(r\ge r_0+1\), \(\chi=1\), so \((\log k)'=\frac{S'}{2S}-\frac1r\) and \(k=b\sqrt S/r\) for some \(b>0\). The formula for \(m\) follows. Finally \(rk(r)=b\sqrt{S(r)}\to\infty\), since \(S(r)\to\infty\) by~\eqref{eq:growth_B}.
\end{proof}

We define the bilinear form
\[
g_k\coloneqq d(k\beta)(\cdot,J\cdot).
\]

\begin{lemma}\label{lem:gk}
\(g_k\) is a Riemannian metric, i.e. \(d(k\beta)\) is a \(J\)-compatible symplectic form, and
\[
m\,g\le g_k\le k\,g .
\]
\end{lemma}
\begin{proof}
On \(\{R\le r_0\}\) we have \(k=1\), so \(g_k=g\). On \(\{R>r_0\}\),
\( d(k\beta) = k d\beta -k'\,dR\wedge(dR\circ J) \). Both terms are \(J\)-invariant, so \( d( k \beta) ( J \cdot, J\cdot) = d(k \beta)( \cdot, \cdot) \), the form \(g_k\) is symmetric, and
\[
g_k(v,v)
=k |v|_g^2
+k' \left(dR(v)^2+dR(Jv)^2\right).
\]
Write \(dR(v)=g(Z_\beta,v)\) and \(dR(Jv)=-g(JZ_\beta,v)\), where \(Z_\beta\) and \(JZ_\beta\) are orthogonal and of the same length. Then \( 0\le dR(v)^2 + dR(Jv)^2 \le |dR|_g^2|v|^2_g \le R|v|_g^2\) by~\eqref{eq:a}. Since \(k'\le0\), this gives \(mg\le g_k\le kg\). As \(m>0\), the form \(g_k\) is positive definite.
\end{proof}

Recall that \(\mu_t = \alpha + t k \beta\). By Lemma~\ref{lem:gk} and the weak taming assumption~\eqref{eq:weak_tame},
\[
d\mu_t(v,Jv)=d\alpha(v,Jv)+t\,g_k(v,v)\ge t\,m\,|v|_g^2,
\]
so \(d\mu_t\) tames \(J\) for \(t\in(0,1]\) and is in particular nondegenerate; for \(t=0\), \(d\mu_0=d\alpha\) is nondegenerate by assumption. Define the Moser vector field \(X_t\) and the (time independent) vector field \(W\) by
\[
\iota_{X_t} d \mu_t = - k\beta, \qquad
\iota_W  d\alpha = - k\beta.
\]

\begin{lemma}\label{lem:XW}
\(|X_t|_{g_k}\le |W|_{g_k}\) for \(0\le t\le1\).
\end{lemma}
\begin{proof}
This is a pointwise argument.
Define a linear operator \(T\) by
\(
g_k(Tv,w)=d\alpha(v,Jw).
\)
It is invertible since \(d\alpha\) is nondegenerate, and \(g_k(Tv,v)=d\alpha(v,Jv)\ge0\) by weak taming. For any \(v\),
\[
g_k( (T+t) X_t, v) = d \alpha(X_t, Jv) + t\, d(k\beta)(X_t, Jv) = \iota_{X_t} d\mu_t (Jv) = - k \beta (Jv) = \iota_W d\alpha (Jv) = g_k(TW, v),
\]
which implies
\(
(T+t)X_t=TW
\), i.e. \(W=X_t+tT^{-1}X_t\).
Hence, with \(w=T^{-1}X_t\),
\[
|W|_{g_k}^2
=|X_t|_{g_k}^2
+2t\,g_k(Tw,w)
+t^2|w|_{g_k}^2
\ge |X_t|_{g_k}^2. \qedhere
\]
\end{proof}

\begin{lemma}
The Moser flow \(\phi^1_t\) of \(\mu_t=\alpha+tk\beta\) exists for all \(t\in[0,1]\). In particular, there is a smooth function \(H^1\) with
\[
(\phi^1_1)^*(\alpha+k\beta)=\alpha+dH^1 .
\]
\end{lemma}
\begin{proof}
We first bound \(W\), which does not depend on \(t\). On \(\{R\ge r_0\}\), the definition of \(\norm{\omega^{-1}}_g\), the assumption~\eqref{eq:growth_B} and~\eqref{eq:a} give
\[
|W|_g\le\norm{\omega^{-1}}_g\,|k\beta|_g\le C\,B(R)\,k\,\sqrt R .
\]
Combining with Lemmas~\ref{lem:XW} and~\ref{lem:gk},
\[
|X_t|_g\le\frac1{\sqrt m}|X_t|_{g_k}\le\frac1{\sqrt m}|W|_{g_k}\le\sqrt{\frac km}\,|W|_g\le C\,B(R)\,\sqrt R\,\frac{k^{3/2}}{\sqrt m},
\]
and hence, using \(|dR|_g\le\sqrt R\) once more,
\[
|X_tR|\le|dR|_g\,|X_t|_g\le C\,B(R)\,R\,\frac{k^{3/2}}{\sqrt m}.
\]
On \(\{R\ge r_0+1\}\) the explicit formulas of Lemma~\ref{lem:k} give \(k^{3/2}/\sqrt m=\sqrt2\,b\,B(R)S(R)/R^2\). Therefore
\[
|X_tR|\le\sqrt2\,Cb\,\frac{B(R)^2S(R)}R,
\qquad\text{so that}\qquad
|X_t\log S(R)|=\frac{S'(R)}{S(R)}\,|X_tR|\le\sqrt2\,Cb,
\]
because \(S'(r)=r/B(r)^2\). On the compact set \([0,1]\times\{R\le r_0+1\}\) the smooth function \(X_t\log S(R)\) is bounded anyway. Finally, \(S\) is increasing with \(S(r)\to\infty\), so \(P=\log S(R)\) is a proper function, bounded below by \(0\), with \(|X_tP|\) bounded uniformly in \(t\). By Lemma~\ref{lem:complete}, the flow \(\phi^1_t\) exists for \(t\in[0,1]\), and~\eqref{eq:moser_exact} gives the identity.
\end{proof}
This is the only place where the integral condition in~\eqref{eq:growth_B} is used: it is exactly what makes \(S\), hence \(P\), unbounded.

\subsection{The second path}
We now choose the large coefficient \(A\). Since \(R\) is proper, the numbers
\[
M_n\coloneqq1+\sup_{\{R\le r_0+n+1\}}|\alpha|_g,\qquad n=0,1,2,\dots
\]
are finite and nondecreasing in \(n\). Let \(\vartheta:\bR\to[0,1]\) be a smooth nondecreasing function with \(\vartheta=0\) on \((-\infty,0]\) and \(\vartheta=1\) on \([1,\infty)\), and set
\[
A(r)\coloneqq M_0+\sum_{j\ge1}(M_j-M_{j-1})\,\vartheta(r-r_0-j+1).
\]
The sum is locally finite, so \(A\) is smooth and nondecreasing, and \(A=M_0\) on \([1,r_0]\). If \(r\in[r_0+n,r_0+n+1]\), the first \(n\) summands have reached their full value, so \(A(r)\ge M_n\ge1+|\alpha_x|_g\) for every \(x\) with \(R(x)=r\). Hence
\begin{equation}\label{eq:A_props}
A\ge1\ge k,\qquad A'\ge0,\qquad A(R(x))\ge1+|\alpha_x|_g\quad\text{for all }x\in X.
\end{equation}
No upper bound on the growth of \(A\) is imposed. This is why Theorem~\ref{thm:bound} needs no growth condition on \(\alpha\), only on \(\norm{(d\alpha)^{-1}}_g\).

The path from \(\alpha+k\beta\) to \(\alpha+A\beta\) is a special case of the following lemma, which we will use again for the fourth path. For a smooth function \(\rho:[1,\infty)\to(0,\infty)\) which is constant on \([1,r_0]\), write
\[
g_\rho\coloneqq d(\rho(R)\beta)(\cdot,J\cdot),
\]
so that \(g_k\) is the form of the first path, and \(g_1=g\).

\begin{lemma}\label{lem:radial}
We have the following: 
\begin{enumerate}[label=\textup{(\alph*)}]
\item Suppose \(\rho+r\rho'>0\) on \([1,\infty)\). Then \(g_\rho\) is a Riemannian metric, i.e. \(d(\rho(R)\beta)\) is a \(J\)-compatible symplectic form. If \(\rho'\le0\) then \((\rho+R\rho')g\le g_\rho\le\rho g\), and if \(\rho'\ge0\) then \(g_\rho\ge\rho g\). Moreover, on \(\{R>r_0\}\), writing \(a\coloneqq|dR|_g^2\le R\) and \(D_\rho\coloneqq\rho(R)+a\,\rho'(R)\), we have \(D_\rho>0\), and the dual norm of \(dR\) is
\begin{equation}\label{eq:dual_dR}
|dR|_{g_\rho}^2=\frac a{D_\rho}.
\end{equation}
\item Let \(\sigma\) be a one-form on \(X\) with \(d\sigma(v,Jv)\ge0\) for all \(v\) (we allow \(d\sigma\) to be degenerate, e.g. \(\sigma=0\)). Let \(\rho_0\le\rho_1\) be two functions as in \textup{(a)}, and assume in addition that \(r\rho_0(r)\to\infty\) as \(r\to\infty\). Put \(\rho_u=(1-u)\rho_0+u\rho_1\) and
\[
\mu_u=\sigma+\rho_u(R)\,\beta,\qquad u\in[0,1].
\]
Then \(d\mu_u\) is nondegenerate, and the Moser flow of the path \(\mu_u\) exists for all \(u\in[0,1]\). The same holds for the reversed path \(u\mapsto\mu_{1-u}\).
\end{enumerate}
\end{lemma}

\begin{proof}
(a) On \(\{R\le r_0\}\), \(\rho'=0\) and \(g_\rho=\rho g\). On \(\{R>r_0\}\), exactly as in the proof of Lemma~\ref{lem:gk},
\[
g_\rho(v,w)=\rho\,g(v,w)+\rho'\bigl(dR(v)\,dR(w)+dR(Jv)\,dR(Jw)\bigr).
\]
Let \(\Pi\subset T_xX\) be the plane spanned by \(Z_\beta=\nabla_gR\) and \(JZ_\beta\). These two vectors are \(g\)-orthogonal and have the same length \(\sqrt a\). Since \(dR(v)=g(Z_\beta,v)\) and \(dR(Jv)=-g(JZ_\beta,v)\), the correction term vanishes as soon as \(v\) or \(w\) lies in \(\Pi^\perp\), and equals \(a\rho'\,g(v,w)\) for \(v,w\in\Pi\). Thus \(g_\rho=D_\rho\,g\) on \(\Pi\), \(g_\rho=\rho\,g\) on \(\Pi^\perp\), and \(\Pi\perp\Pi^\perp\) for both forms. (If \(a=0\), simply \(g_\rho=\rho g\).) Now \(D_\rho\ge\rho>0\) if \(\rho'\ge0\), and \(D_\rho\ge\rho+R\rho'>0\) if \(\rho'<0\), because \(a\le R\). This proves that \(g_\rho\) is positive definite, as well as the two-sided bounds. Finally, the covector \(dR=g(Z_\beta,\cdot)\) vanishes on \(\Pi^\perp\), so its dual norm is computed on \(\Pi\), where \(g_\rho=D_\rho g\): \(|dR|^2_{g_\rho}=|dR|_g^2/D_\rho=a/D_\rho\).

(b) Put \(\delta=\rho_1-\rho_0\ge0\), so that \(\dot\mu_u=\delta\beta\), and write \(g_u=g_{\rho_u}\), \(D_u=D_{\rho_u}\). Part (a) applies to \(\rho_u\), because \(\rho_u+r\rho_u'=(1-u)(\rho_0+r\rho_0')+u(\rho_1+r\rho_1')>0\). Together with the assumption on \(\sigma\), this gives
\[
d\mu_u(v,Jv)=d\sigma(v,Jv)+|v|_{g_u}^2>0\qquad(v\ne0).
\]
Hence \(d\mu_u\) is nondegenerate, and the Moser field \(X_u\), defined by \(\iota_{X_u}d\mu_u=-\delta\beta\), is smooth in \((u,x)\). On \(\{R>r_0\}\) we have \(\beta(JX_u)=-dR(J^2X_u)=X_uR\), so~\eqref{eq:energy} reads
\[
|X_u|_{g_u}^2+d\sigma(X_u,JX_u)=-\delta\,X_uR .
\]
The left hand side is \(\ge|X_u|^2_{g_u}\). Thus \(X_uR\le0\) wherever \(\delta>0\), while \(X_u=0\) wherever \(\delta=0\). Setting \(y=-X_uR\ge0\) and using Cauchy--Schwarz together with~\eqref{eq:dual_dR},
\[
y^2=dR(X_u)^2\le|dR|^2_{g_u}|X_u|^2_{g_u}\le\frac a{D_u}\,\delta\,y,
\]
so that, on \(\{R>r_0\}\),
\begin{equation}\label{eq:radial_bound}
0\le-X_uR\le\frac{\delta\,a}{D_u}.
\end{equation}
The first inequality says that \(R\) does not increase along trajectories in \(\{R>r_0\}\). This controls the flow in forward time. For backward time, consider \(Q_u(r)\coloneqq r\rho_u(r)\). Along a trajectory \(x(u)\) in \(\{R>r_0\}\), by~\eqref{eq:radial_bound}, \(\rho_u+R\rho_u'>0\) and \(a\le R\),
\[
\frac d{du}\,Q_u(R(x(u)))
=\delta R+(\rho_u+R\rho_u')\,X_uR
\ge\delta R-(\rho_u+R\rho_u')\frac{\delta a}{D_u}
=\frac{\delta\,\rho_u\,(R-a)}{D_u}\ge0 .
\]
So \(Q_u(R(x(u)))\) is nondecreasing in \(u\). Moreover \(Q_u(r)\ge r\rho_0(r)\eqqcolon P_0(r)\), where \(P_0(r)\to\infty\) by assumption.

On the compact set \([0,1]\times\{R\le r_0\}\), the two smooth functions \(X_uR\) and \((\partial_u+X_u)(Q_u(R))\) are bounded, say \(X_uR\le C_+\) and \((\partial_u+X_u)(Q_u(R))\ge-C_-\). By the above, these two inequalities then hold on all of \(X\). Consequently, a trajectory with \(x(s)=x\) satisfies
\[
R(x(u))\le R(x)+C_+\quad(u\ge s),\qquad
P_0(R(x(u)))\le Q_u(R(x(u)))\le Q_s(R(x))+C_-\quad(u\le s).
\]
Since \(R\) is proper and \(P_0(r)\to\infty\), the trajectory stays in a compact set in both time directions, and Lemma~\ref{lem:complete} applies. The reversed path \(\mu_{1-u}\) has Moser field \(-X_{1-u}\), whose trajectories are the time reversals of those of \(X_u\). So they exist on \([0,1]\) as well.
\end{proof}

We apply Lemma~\ref{lem:radial}(b) with \(\sigma=\alpha\), \(\rho_0=k\) and \(\rho_1=A\). The hypotheses hold: \(d\alpha(v,Jv)\ge0\) by~\eqref{eq:weak_tame}; \(m>0\) and \(rk(r)\to\infty\) by Lemma~\ref{lem:k}; \(A+rA'\ge A>0\) and \(k\le1\le A\) by~\eqref{eq:A_props}. We obtain the Moser flow \(\phi^2_u\) of the path \(\mu_u=\alpha+\bigl((1-u)k+uA\bigr)\beta\), and
\[
(\phi^2_1)^*(\alpha+A\beta)=\alpha+k\beta+dH^2 .
\]

\subsection{The third path}
Consider \(\mu_s=(1-s)\alpha+A\beta\), \(s\in[0,1]\), so that \(\dot\mu_s=-\alpha\) and \(\iota_{X_s}d\mu_s=\alpha\). Lemma~\ref{lem:radial}(a) gives \(g_A\ge Ag\), because \(A'\ge0\). With weak taming,
\[
d\mu_s(v,Jv)=(1-s)\,d\alpha(v,Jv)+|v|_{g_A}^2\ge A\,|v|_g^2 ,
\]
so \(d\mu_s\) is nondegenerate. Evaluating on \(JX_s\) as in~\eqref{eq:energy},
\[
A\,|X_s|_g^2\le d\mu_s(X_s,JX_s)=\alpha(JX_s)\le|\alpha|_g\,|X_s|_g,
\qquad\text{so}\qquad
|X_s|_g\le\frac{|\alpha|_g}{A}\le1
\]
by~\eqref{eq:A_props}. This is where the size of \(A\) is used. Consequently \(|X_sR|\le|dR|_g\le\sqrt R\) on \(\{R\ge r_0\}\), i.e. \(|X_s\sqrt R|\le\tfrac12\) there, and \(X_s\sqrt R\) is bounded on the compact set \([0,1]\times\{R\le r_0\}\). Lemma~\ref{lem:complete} with \(P=\sqrt R\) gives the flow \(\phi^3_s\), and
\[
(\phi^3_1)^*(A\beta)=\alpha+A\beta+dH^3 .
\]

\subsection{The fourth path}
Finally, let \(\mu_v=\bigl((1-v)A+v\bigr)\beta\), \(v\in[0,1]\), which goes from \(A\beta\) to \(\beta\). This is the reversed path in Lemma~\ref{lem:radial}(b) with \(\sigma=0\), \(\rho_0=1\) and \(\rho_1=A\); the hypotheses are clear, since \(r\rho_0(r)=r\to\infty\). We obtain the flow \(\phi^4_v\), and
\[
(\phi^4_1)^*\beta=A\beta+dH^4 .
\]
Composing the four paths completes the proof of Theorem~\ref{thm:bound}. \qed

\section{Proof of Corollary \ref{cor: quadratic poisson}} \label{sec: quadratic poisson}

We choose $R$ as in Lemma \ref{lem:stein}.
Let $\Pi_\Omega$ be the holomorphic Poisson bivector of $\Omega$.
Its pushforward along $z$ has components
$\{z_\mu,z_\nu\}_\Omega$ in the ambient coordinates.
Since $z$ is isometric for the induced metric, the quadratic
hypothesis gives
\[
 |\Pi_\Omega|_g
 \le C\left(1+\sum_{\nu=1}^N|z_\nu|^2\right)
 \le C'R.
\]
The real bivector $(\operatorname{Im}\Omega)^{-1}$ is obtained
from $\Pi_\Omega$ by a fixed real-linear operation, so
$
 \|(\operatorname{Im}\Omega)^{-1}\|_g\le C''R$.
Since $\Omega$ is holomorphic symplectic, for every real tangent vector $v$,
$\Omega(v,Iv)=i\Omega(v,v)=0.$
Hence $\operatorname{Im}\Omega$ is weakly tamed by $I$. The the result follows from Theorem \ref{thm:bound}, by taking \(B(r) = r\).
\qed

\begin{remark}
    Brackets of degree $d$ would give only $\norm{(\Im\Omega)^{-1}}_g \le C R^{d/2}$, and $B(r) = r^{d/2}$ satisfies \eqref{eq:growth_B} only for $d \le 2$.
\end{remark}

\section{Examples}\label{sec:applications}
There are several instances in the mirror symmetry literature where one is a priori interested in $\Fuk(Y, \lambda)$, but the authors compute instead $\Fuk(Y', \lambda')$ for some $(Y', \lambda')$ which is not checked to be Liouville deformation equivalent to the original space of interest.  This is typically because the original $Y$ is an affine algebraic variety and $\lambda$ the corresponding K\"ahler Liouville  form, whereas $\lambda'$ is chosen by hand to ensure the Lagrangian skeleton of $(Y', \lambda')$ is computable.  Here we mention a couple examples, and show in these cases, as a consequence of our results, that the corresponding Fukaya categories nevertheless agree.

\subsection{Very affine hypersurfaces}

\label{eg:hypersurface}
Let $H \subset (\bC^*)^n$ be a smooth algebraic hypersurface.  In \cite{Gammage-Shende, Zhou} the Lagrangian skeleton of $H$ is computed by way of a `tropical localization' $H_t$, $t \in [0,1]$, $H_0 = H$, constructed by the methods of \cite{Mikhalkin, Abouzaid}.  The symplectic form is the restriction of a toric K\"ahler form $\omega_\varphi = d\lambda_\varphi$, $\lambda_\varphi = -d\varphi \circ J$, for a potential $\varphi = \varphi(\log|z_1|, \dots, \log|z_n|)$ with
\begin{equation}\label{eq:hess}
\operatorname{Hess}\varphi \ge c\,I \qquad \text{for some } c > 0;
\end{equation}
for instance $\varphi = \frac12 \sum_i (\log|z_i|)^2$ in \cite{Gammage-Shende}, or the homogeneous potentials of \cite[Def.~2.8]{Zhou}.  The hypersurfaces $H_t$ are symplectic for $\omega_\varphi$ and symplectomorphic to each other \cite[Prop.~1.7, 1.8]{Zhou}, and the skeleton of $(H_1, \lambda_\varphi|_{H_1})$ is determined in \cite[Thm.~2]{Zhou}.  However, the symplectomorphism $\psi: (H, \omega_\varphi|_H) \to (H_1, \omega_\varphi|_{H_1})$ is not shown there to be a Liouville deformation; see \cite[Rem.~5.3.3]{Gammage-Shende}.  
More recently, the fact that it is a Liouville deformation has been shown in \cite{fltz} and formalized in Lean.

Here we note that Theorem~\ref{thm:bound} gives a direct route to the comparison of Fukaya categories.  Take $\alpha = \psi^*(\lambda_\varphi|_{H_1})$, a Liouville form of finite type on $H$ with $d\alpha = \omega_\varphi|_H$, and let $J$, $\beta$, $R$ and $g$ be as in Lemma~\ref{lem:stein} for the embedding $\iota: H \to \bC^{2n}$, $z \mapsto (z, z^{-1})$.  Then $d\alpha$ tames $J$ and $\norm{(d\alpha)^{-1}}_g \le \frac4c\,R$. 
So Theorem~\ref{thm:bound} applies with $B(r) = r$, and
\[
\Fuk(H, \beta) \cong \Fuk(H, \alpha) \cong \Fuk(H_1, \lambda_\varphi|_{H_1}).
\]

\subsection{Multiplicative hypertoric varieties}

We recall the multiplicative hypertoric varieties for which homological mirror symmetry was established in~\cite{Gammage-McBreen-Webster}.  

Let $T \subset (\bC^*)^n$ be a unimodular subtorus of dimension $k$ with weight matrix $A \in \operatorname{Mat}_{k \times n}(\mathbb Z)$, acting on
\[
Y = \{ (z, w) \in \bC^{2n} : 1 + z_i w_i \ne 0 \text{ for all } i \}, \qquad \Omega_Y = \sum_i \frac{dz_i \wedge dw_i}{1 + z_i w_i},
\]
with weights $A$ on the $z_i$ and $-A$ on the $w_i$.  The action is quasi-Hamiltonian with group valued moment map $\mu_\ell = \prod_i (1 + z_i w_i)^{A_{\ell i}}$, $1 \le \ell \le k$, and for generic $\tau \in (S^1)^k$ the multiplicative hypertoric variety $M_B = \mu^{-1}(\tau) /\!/ T$ is a smooth affine variety with holomorphic symplectic form $\Omega_B$ and complex structure $J_B$; it is the variety $\mathfrak U_{(\tau,1)}$ of \cite[Def.~2.3]{Gammage-McBreen-Webster}.

In \cite{Gammage-McBreen-Webster}, the wrapped Fukaya category of $M_B$ is computed for a Liouville structure $\lambda_{Dol}$ obtained as follows.  A hyperk\"ahler `Dolbeault' space $M_{Dol}$ is built by reduction from Ooguri--Vafa metrics on elliptic fibrations \cite[\S 7]{DMS}, \cite[\S 3]{Gammage-McBreen-Webster}.  One of its K\"ahler forms, $\omega_J$ in the notation of \cite[Thm.~3.6]{Gammage-McBreen-Webster}, admits a Liouville domain $D \subset M_{Dol}$ whose skeleton $L$ is the central fiber, a holomorphic Lagrangian \cite[Lem.~3.36, 3.37]{Gammage-McBreen-Webster}.  The completion $\widehat D$ is identified with $M_B$ by a diffeomorphism \cite[Prop.~3.46]{Gammage-McBreen-Webster}, constructed in \cite[Lem.~11.1.1]{DMS} by matching moment maps, and the resulting Liouville manifold is weakly Weinstein \cite[Lem.~3.47]{Gammage-McBreen-Webster}.  The relation between $\lambda_{Dol}$ and the K\"ahler Liouville form of the affine variety $M_B$ was however left open  \cite[Warning~3.2, Rem.~3.48]{Gammage-McBreen-Webster}.

We show that, for the K\"ahler Liouville form $\beta$ of the closed embedding $\iota: M_B \hookrightarrow \bC^N$ of Lemma~\ref{lem:ht-quadratic} below,
\begin{equation}\label{eq:ht-claim}
\Fuk(M_B, \beta) \cong \Fuk(\widehat D, \lambda_{Dol}).
\end{equation}

\reversemarginpar
\begin{lemma}\label{lem:ht-quadratic}
There is a closed embedding $\iota: M_B \hookrightarrow \bC^N$ whose coordinate functions have Poisson brackets, with respect to $\Omega_B$, which are polynomials of degree at most two in these coordinates.  For this embedding, the form $\beta$ of Lemma~\ref{lem:stein} is Weinstein.
\end{lemma}
\normalmarginpar

\begin{proof}
Change the basis of the character lattice of $T$ so that $A = [I_k \mid A_0]$.  Unimodularity says that every nonzero maximal minor of $A$ is $\pm1$.  A smaller nonzero minor of $A$ can be completed to a maximal minor by columns of $I_k$, so every minor of $A$, and hence every minor of $M = [A \mid -A]$, is $0$ or $\pm 1$.
Let $S = \ker M \cap \{0,1\}^{2n}$.  

We first show that if $a \in \ker M \cap \mathbb Z_{\ge 0}^{2n}$ has all entries at most $q$, then $a$ is a sum of $q$ elements of $S$, zero summands allowed.  We argue by induction on $q$.  The polytope
\[
P = \{ v \in \bR^{2n} : Mv = 0,\ \max(0, a_j - q + 1) \le v_j \le \min(1, a_j) \}
\]
contains $a/q$, so it is nonempty, and its vertices are integral because the minors of $M$ are $0, \pm 1$.
An integral vertex $v$ lies in $S$, and $a - v \in \ker M \cap \mathbb Z_{\ge 0}^{2n}$ has entries at most $q - 1$, qed.

Applying this to rational points after clearing denominators, and using that finitely generated cones are closed, we also get
\begin{equation}\label{eq:ht-cone}
\ker M \cap \bR_{\ge 0}^{2n} = \sum_{s \in S \setminus 0} \bR_{\ge 0}\, s .
\end{equation}

For $s = (p, r) \in S$ put $m_s = z^p w^r$; these are the squarefree $T$-invariant monomials, and by the decomposition they generate $\bC[z, w]^T$.  Adjoin $u_i = (1 + z_i w_i)^{-1}$. 
Localization at invariant functions commutes with invariants, and $T$ is linearly reductive, so the functions $m_s$, $s \ne 0$, and $u_i$ generate $\bC[Y]^T$ and then $\bC[M_B]$.  Let $\iota$ be the closed embedding they define.

The Poisson bracket of $\Omega_Y$ is $\{z_i, w_j\} = \delta_{ij}(1 + z_i w_i)$.  Differentiating, $\{m_s, u_i\} = (r_i - p_i)\, m_s u_i$, $\{u_i, u_j\} = 0$, and for $s = (p, r)$, $t = (p', r')$, $c_i = p_i r'_i - r_i p'_i$,
\begin{equation}\label{eq:ht-bracket}
\{m_s, m_t\} = \Bigl( \sum_i c_i \Bigr) m_s m_t + \sum_{i : c_i \ne 0} c_i\, \frac{m_s m_t}{z_i w_i} .
\end{equation}
If $c_i \ne 0$, the exponent vector of $m_s m_t / (z_i w_i)$ is $s + t - e_i - e_{n+i}$, which lies in $\ker M \cap \mathbb Z_{\ge 0}^{2n}$ with entries at most two, so it is a sum of two elements of $S$.  Hence all brackets are polynomials of degree at most two in the generators.  They descend to $M_B$, because $\{\mu_\ell, m_s\} = \mu_\ell m_s \sum_i A_{\ell i}(r_i - p_i) = 0$ and $\{\mu_\ell, u_i\} = 0$, so that the moment map equations generate a Poisson ideal of the invariant ring, and $\Omega_B$ is the reduced form.

Translating $\iota$ by a generic vector preserves the quadratic bracket property and makes $R = 1 + \frac14 |\iota|^2$ a Morse function.  Its critical set is a discrete semialgebraic set, hence finite, so $\beta = -dR \circ J_B$ is a Weinstein form.
\end{proof}

\begin{lemma}\label{lem:ht-completion} 
Let $(D, \lambda)$ be a compact Liouville domain and $(M, \omega)$ a symplectic manifold.  Suppose given a smooth family of embeddings $\sigma_c: D \to M$, $c \ge 0$, defined on neighborhoods of $D$, with
\[
\sigma_c^* \omega = e^c\, d\lambda, \qquad E_c := \sigma_c(D) \subset \operatorname{int} E_d \ (c < d), \qquad \bigcup_{c \ge 0} E_c = M .
\]
Suppose moreover that $D$ contains a compact Whitney stratified subset $L$ such that $H_1(L; \bR) \to H_1(D; \bR)$ is onto, $\lambda$ vanishes on the tangent spaces of the strata of $L$, and for each stratum $S$ the map $[0, \infty) \times S \to M$, $(c, x) \mapsto \sigma_c(x)$, takes values in a fixed isotropic submanifold of $M$.  Then $(M, \omega)$ is symplectomorphic to the completion $(\widehat D, d\widehat\lambda)$.
\end{lemma}

\begin{proof}
In brief: the hypotheses give symplectomorphisms on exhausting compact domains, not constrained to be compatible with the 1-form.  The condition on $L$ kills the flux of the transition isotopies, allowing compactly supported Hamiltonian corrections that make the maps compatible.

Let $\phi_c$ be the Liouville flow of $\widehat D$, $D_c = \phi_c(D)$ and $\psi_c = \sigma_c \circ \phi_{-c}: D_c \to E_c$.  The $D_c$ exhaust $\widehat D$, and $\psi_c^* \omega = \phi_{-c}^*(e^c\, d\lambda) = d\widehat\lambda$.

We first show that on each $D_a$, the family $\psi_c$, $c \ge a$, is a Hamiltonian isotopy.  Let $\eta_c = \sigma_c^*(\iota_{\partial_c \sigma_c}\, \omega) \in \Omega^1(D)$.  Differentiating $\sigma_c^*\omega = e^c\, d\lambda$ in $c$ gives $d\eta_c = e^c\, d\lambda$.  Both $\eta_c$ and $\lambda$ vanish on the tangent spaces of the strata of $L$: for $\eta_c$, the two vectors paired by $\omega$ are tangent to the same isotropic submanifold.  So the closed form $\eta_c - e^c \lambda$ has zero periods on $L$, hence on $D$, and $\eta_c - e^c\lambda = dh_c$ with $h_c$ smooth in $c$.  Since $\partial_c \psi_c = (\partial_c \sigma_c) \circ \phi_{-c} - d\sigma_c(Z \circ \phi_{-c})$, where $Z$ is the Liouville vector field,
\[
\psi_c^*(\iota_{\partial_c \psi_c}\, \omega) = \phi_{-c}^*(\eta_c - e^c \iota_Z d\lambda) = \phi_{-c}^*(\eta_c - e^c\lambda) = d(h_c \circ \phi_{-c}) .
\]

Fix $a < b$.  The set $\bigcup_{c \in [a, b]} \psi_c(D_a)$ is compactly contained in $\operatorname{int} E_b$: for $c < b$ because $E_c \subset \operatorname{int} E_b$, and for $c = b$ because $D_a \subset \operatorname{int} D_b$.  So the Hamiltonians of the isotopy can be extended and cut off inside $\operatorname{int} E_b$.  Their flow gives a Hamiltonian diffeomorphism $K_{a,b}$ of $M$, supported in $\operatorname{int} E_b$, with $K_{a,b} \circ \psi_a = \psi_b$ on $D_a$.

Choose $0 = c_0 < c_1 < \cdots \to \infty$ and define Hamiltonian diffeomorphisms $H_j$ supported in $\operatorname{int} E_{c_j}$ by $H_0 = \mathrm{id}$ and $H_{j+1} = H_j \circ K_{c_j, c_{j+1}}^{-1}$.  Then $\theta_j = H_j \circ \psi_{c_j}: D_{c_j} \to E_{c_j}$ is a symplectic diffeomorphism, since $H_j$ preserves $E_{c_j}$, and $\theta_{j+1} = \theta_j$ on $D_{c_j}$.  The $\theta_j$ define a symplectic embedding $\Theta: \widehat D \to M$ whose image contains every $E_{c_j}$, so $\Theta$ is onto.
\end{proof}

We now construct such a family for $M_B$.  Recall from \cite[Lem.~3.35]{Gammage-McBreen-Webster} that the Liouville vector field of $M_{Dol}$ satisfies $Z(F) \ge F$ for $F = \sum_i |q_i|^2$, where $q_i = x_i + i y_i$ are the complex moment maps of the circle actions.  So all the sublevel sets $\{F \le r^2\}$ are Liouville domains with the same completion $\widehat D$ and the same skeleton $L = \{F = 0\}$, the central fiber.

\begin{lemma}\label{lem:ht-family}
There is $r > 0$, a constant $\kappa \in \bR^*$, and embeddings $\sigma_c: D = \{F \le r^2\} \to M_B$, $c \ge 0$, defined on neighborhoods of $D$ and smooth in $c$, such that
\begin{equation}\label{eq:ht-conformal}
\sigma_c^* \bigl( \kappa^{-1} \Im\Omega_B \bigr) = e^c\, d\lambda_{Dol},
\end{equation}
the images $E_c = \sigma_c(D)$ satisfy $E_c \subset \operatorname{int} E_d$ for $c < d$ and $\bigcup_c E_c = M_B$, and each $\sigma_c$ maps each stratum of $L$ into a fixed holomorphic Lagrangian stratum of $M_B$.
\end{lemma}

\begin{proof}
\emph{The Ooguri--Vafa space.}  $M_{Dol}$ is built from the Ooguri--Vafa space $\mathfrak Z$ \cite[\S 7.2]{DMS}, which is the Gibbons--Hawking construction on a circle bundle over an open subset of $(\bR/\mathbb Z)_t \times \bR^2_{(x, y)}$ with periodic harmonic function
\[
V_0(t, x, y) = C_0 + \frac{1}{2\sqrt{t^2 + x^2 + y^2}} + \sum_{m \ne 0} \Bigl( \frac{1}{2\sqrt{(t+m)^2 + x^2 + y^2}} - \frac{1}{2|m|} \Bigr),
\]
on the region $V_0 > 0$.  Let $\theta$ be the connection form, so that the three K\"ahler forms are $\omega_1 = dt \wedge \theta + V_0\, dx \wedge dy$ and its cyclic permutations.  We label the complex structures so that $I$ has holomorphic base coordinate $q = x + iy$, and we put
\[
\Omega_I = \omega_J + i\omega_K = i\,(V_0\, dt + i\theta) \wedge dq, \qquad \Omega_K = \omega_I + i\omega_J = i\,(V_0\, dy + i\theta) \wedge d\eta, \qquad \eta = t + ix .
\]
Then $\Omega_I$ is the algebraic holomorphic symplectic form of the elliptic fibration $q: \mathfrak Z \to \bC$ \cite[Lem.~7.2.11]{DMS}, and after reduction $d\lambda_{Dol} = \omega_J$.  Both formulas follow from the Gibbons--Hawking expressions.

\emph{Shifting the potential.}  For $c \ge 0$ let $V_c = V_0 + c$.  Since $dV_c = dV_0$, the same connection $\theta$ serves for all $c$.  Since $V_0 = -\log\sqrt{x^2 + y^2} + O(1)$ uniformly in $t$, there are $r, \delta > 0$ with
\begin{equation}\label{eq:ht-positive}
V_c \ge \delta \quad \text{on } \{ |q| \le e^c r \}, \qquad c \ge 0 .
\end{equation}
Indeed, choose $r_*, \delta$ with $V_0 \ge \delta$ on $|q| \le r_*$, then $C$ with $V_0 \ge -\log|q| - C$ for $|q| \ge r_*$, then $r < r_*$ with $-\log r - C \ge \delta$; for $|q| \le r_*$ use $V_c \ge V_0$, and otherwise $V_c \ge c - \log|q| - C \ge -\log r - C$.  Let $\mathfrak Z_c$ be the Ooguri--Vafa space with potential $V_c$ over $|q| < e^c r$.

On the universal cover of the smooth locus, $u_c = \exp \int (V_c\, dt + i\theta)$ is an $I$-holomorphic fiber coordinate, and the deck transformation $t \mapsto t + 1$ multiplies it by $\exp(\int_0^1 V_c\, dt + i \int_0^1 \theta)$.  So the fiber over $q$ is the Tate curve with parameter $p_c(q)$, $\log|p_c(q)| = -\int_0^1 V_c(t, x, y)\, dt = \log|p_0(q)| - c$ and $\arg p_c(q) = \arg p_0(q)$; thus $p_c = e^{-c} p_0$, and $p_0(q) = q$ up to a constant \cite[\S 7.1]{DMS}.  Hence
\begin{equation}\label{eq:ht-tate}
\Phi_c: \mathfrak Z_0 \to \mathfrak Z_c, \qquad (q_{Dol}, u_0) \longmapsto (q, u_c) = (e^c q_{Dol}, u_0)
\end{equation}
is a well defined biholomorphism for $I$ over $\{|q_{Dol}| < r\}$, which extends over the nodal fiber in the plumbing charts.  Now $\Omega_{I,c} = i\, d\log u_c \wedge dq$: both sides are $I$-holomorphic $2$-forms whose contraction with the generator of the circle action is $i\,dq$, because $q$ is the $(\omega_J + i\omega_K)$-moment map.  Therefore
\begin{equation}\label{eq:ht-scaling}
\Phi_c^* \Omega_{I,c} = e^c\, \Omega_{I,0}, \qquad \Phi_c^* \omega_{J,c} = e^c\, \omega_{J,0} .
\end{equation}

\emph{Li's coordinates.}  By \cite[Thm.~B.0.1]{DMS}, see also \cite[Thm.~3.8]{Gammage-McBreen-Webster}, $\mathfrak Z_c$ has a $K$-holomorphic embedding $j_c$ into $\{1 + zw \ne 0\} \subset \bC^2$ with $1 + z w = e^{2\pi i \eta}$, equivariant for the circle actions, with
\[
d\log z = V_c\, dy + i\theta + F(\eta, y)\, d\eta, \qquad d\log w = -V_c\, dy - i\theta + G(\eta, y)\, d\eta,
\]
where $F$ and $G$ do not depend on $c$ (they are  determined by the singularities of $V_0$), and the periods of the right hand sides are normalized to be zero or in $2\pi i \mathbb Z$.  At fixed $(t,x,y)$, we keep the same $\theta$ since $dV_c=dV_0$, and the explicit $F,G$ above are unchanged.
Thus $d\log z_c-d\log z_0=c\,dy$ and $d\log w_c-d\log w_0=-c\,dy$.
These exact differences leave the periods unchanged. Choosing the integration constants compatibly gives
\begin{equation}\label{eq:ht-li}
z_c = e^{cy} z_0, \qquad w_c = e^{-cy} w_0, \qquad 1 + z_c w_c = e^{2\pi i \eta} .
\end{equation}
We see $d\log z \wedge d\log w = (F + G)\,(V_c\, dy + i\theta) \wedge d\eta$, while $(F + G)\, d\eta = d\log(zw) = 2\pi i\, e^{2\pi i\eta}\, d\eta / (zw)$.  Hence
\begin{equation}\label{eq:ht-rotated}
j_c^* \frac{dz \wedge dw}{1 + zw} = \frac{zw}{1 + zw}\, j_c^*(d\log z \wedge d\log w) = 2\pi i\,(V_c\, dy + i\theta) \wedge d\eta = \kappa\, \Omega_{K, c}, \qquad \kappa = 2\pi .
\end{equation}

\emph{Reduction.}  The Dolbeault space is the quasi-hyperhamiltonian reduction of $\mathfrak Z^n$ by $T_{\bR}$ \cite[Def.~3.21]{Gammage-McBreen-Webster}; in the coordinates above its level set is $\{ A x = 0,\ A y = 0,\ e^{2\pi i A t} = \tau \}$.  The identifications $\Phi_c^n$ preserve this level set and the $T_{\bR}$-action, so they descend to the reduced spaces, and \eqref{eq:ht-scaling} descends as well.  On the Betti side, $\mu_\ell = \prod_i (1 + z_i w_i)^{A_{\ell i}} = e^{2\pi i (A(t + ix))_\ell} = \tau_\ell$ on the level set, so $j_c^n$ maps the level set into $\mu^{-1}(\tau)$, equivariantly for $T_{\bR} \subset T$.  Composing, we obtain smooth maps
\[
\sigma_c = j_c^n \circ \Phi_c^n : D = \{ F \le r^2 \} \longrightarrow M_B ,
\]
defined on neighborhoods of $D$ by \eqref{eq:ht-positive}.  By \eqref{eq:ht-scaling} and \eqref{eq:ht-rotated}, $\sigma_c^*(\kappa^{-1}\Im\Omega_B) = \sigma_c^* (\kappa^{-1} \Im \Omega_Y)|_{\text{reduced}} = e^c\, \omega_{J,0} = e^c\, d\lambda_{Dol}$, which is \eqref{eq:ht-conformal}.

\emph{Injectivity.}  Let $W = \ker(A: \bR^n \to \bR^k)$ and $P_W$ the orthogonal projection.  On the locus where all $z_i w_i \ne 0$, the functions $t_i, x_i$ are determined by $1 + z_i w_i$, and the class of $(\log|z_i|)_i$ modulo the positive real part of $T$ is recorded by $b = P_W (\log|z_i|)_i \in W$.  By \eqref{eq:ht-li}, for points of the level set,
\begin{equation}\label{eq:ht-logmap}
b = P_W\, \ell_0(y) + c\,y, \qquad y \in W, \qquad \partial \ell_{0,i} / \partial y_i = V_0(t_i, x_i, y_i),
\end{equation}
where $\ell_{0,i} = \log|z_{0,i}|$ depends only on $(t_i, x_i, y_i)$.  The RHS $y$-derivative is $H_c = P_W \operatorname{diag}(V_c(t_i, x_i, y_i))|_W$, which is positive definite by \eqref{eq:ht-positive}.  So $y \mapsto b$ is injective on the convex set $\{y \in W : |x|^2 + |y|^2 \le e^{2c} r^2\}$.  So two points of $D$ with the same image in $M_B$ lie in the same $T_{\bR}$-orbit.  On the strata where some $z_i$ or $w_i$ vanish, the same argument applies with $-\log|w_i|$ in place of $\log|z_i|$ when $z_i = 0$ (then $y_i \le 0$), and with $y_i = 0$ when both vanish.  So each $\sigma_c$ is injective, and an embedding since it is an immersion by \eqref{eq:ht-conformal}.

\emph{Nesting.}  Fix a point of $M_B$ in the open stratum and let $y(c) \in W$ be its moment coordinate in $E_c$, when it exists; $x$ does not depend on $c$.  Differentiating \eqref{eq:ht-logmap} at fixed $b$ gives $H_c\, y'(c) = -y(c)$, so
\[
\frac{d}{dc} |y(c)|^2 = -2 \langle H_c^{-1} y(c), y(c) \rangle \le 0 .
\]
In  moment coordinates, $E_c$ is the region $|x|^2 + |y|^2 \le e^{2c} r^2$.  The left side is nonincreasing in $c$ and the right side increasing, and the bound on $|y(c)|$ allows the solution $y(c)$ to be continued over every finite interval.  Hence $E_c \subset \operatorname{int} E_d$ for $c < d$; this extends to the other strata by continuity, since the inequality is strict.

\emph{Exhaustion.}  Fix a compact set $K \subset M_B$.  The invariant functions $1 + z_i w_i$, their inverses, and the $m_s$ of Lemma~\ref{lem:ht-quadratic} are bounded on $K$; in particular $x_i = -\frac{1}{2\pi} \log|1 + z_i w_i|$ is bounded on $K$.  For a point of $E_c \cap K$, take its representative on the level set, write $y_i^\pm = \max(\pm y_i, 0)$, and note that $y_i > 0$ forces $z_i \ne 0$ and $y_i < 0$ forces $w_i \ne 0$, since $y_i$ is a monotone function of $|z_i|^2 - |w_i|^2$ along the fibers \cite[Thm.~B.0.1]{DMS}.  We claim
\begin{equation}\label{eq:ht-coercive}
\sum_i \bigl( y_i^+ \log|z_i| + y_i^- \log|w_i| \bigr) \ge \delta\, |y|^2 - C\,(|y| + 1),
\end{equation}
with $C$ independent of $c$.  For $y_i \ge 1$, integrate $\partial_{y_i} \log|z_i| = V_c \ge \delta$ from $1$ to $y_i$; the value at $y_i = 1$ is bounded below independently of $c$ by \eqref{eq:ht-li}.  For $0 < y_i \le 1$, the singular term of $V_c$ gives $\log|z_i| \ge \frac12 \log y_i - C$, and $y \log y$ is bounded on $[0, 1]$.  The case $y_i < 0$ is the same with $w_i$.  On the other hand, $(y^+, y^-) \in \ker M \cap \bR^{2n}_{\ge 0}$, so by \eqref{eq:ht-cone} we can write $(y^+, y^-) = \sum_s a_s s$ with $a_s \ge 0$, and then
\[
\sum_i \bigl( y_i^+ \log|z_i| + y_i^- \log|w_i| \bigr) = \sum_s a_s \log|m_s| \le C_K \sum_s a_s \le C_K \|y\|_1 ,
\]
using that every coordinate appearing with positive exponent is nonzero, and $\sum_s a_s \le \sum_s a_s \|s\|_1 = \|y\|_1$.  Together with \eqref{eq:ht-coercive} this bounds $|y|$ on $E_c \cap K$ independently of $c$.  Since $\partial E_c = \{ |x|^2 + |y|^2 = e^{2c} r^2 \}$, we get $\partial E_c \cap K = \emptyset$ for $c$ large.  Finally $M_B$ is connected, being the closure of a torus bundle over the connected set $\{ Q \in (\bC^*)^n : Q^A = \tau \}$, $Q_i = 1 + z_i w_i$.  Joining a point of $M_B$ to $E_0$ by a path and applying the above to the compact image of the path, we see that the point lies in $E_c$ for $c$ large.

\emph{Strata.}  On the central fiber $L = \{F = 0\}$ we have $x = y = 0$.  The maps $\Phi_c$ preserve each component of the central fiber, and by \eqref{eq:ht-li} the coordinates $z_c, w_c$ agree with $z_0, w_0$ where $y = 0$.  So $\sigma_c(S) = \sigma_0(S)$ for each stratum $S$ of $L$, and $\sigma_0(S)$ is contained in a toric stratum of $M_B$ of the form $\{ z_i = 0\ (i \in I),\ w_i = 0\ (i \notin I) \}$, which is a holomorphic Lagrangian for $\Omega_B$.
\end{proof}

\begin{proof}[Proof of \eqref{eq:ht-claim}]
We check the hypotheses of Lemma~\ref{lem:ht-completion} for the family of Lemma~\ref{lem:ht-family}, with $\lambda = \lambda_{Dol}$, $\omega = \kappa^{-1}\Im\Omega_B$, and $L$ the central fiber.  The strata of $L$ are Lagrangian for $d\lambda_{Dol}$ and invariant under the Liouville flow, so $\lambda_{Dol} = \iota_Z d\lambda_{Dol}$ vanishes on their tangent spaces.  The negative Liouville flow carries $D$ into any neighborhood of $L$, by $Z(F) \ge F$, and such a neighborhood retracts onto $L$; as the flow is isotopic to the identity, $H_1(L; \bR) \to H_1(D; \bR)$ is onto.  The last hypothesis is the last assertion of Lemma~\ref{lem:ht-family}.  We obtain a symplectomorphism
$
\Theta: (\widehat D, d\lambda_{Dol}) \longrightarrow (M_B, \kappa^{-1}\Im\Omega_B)$.
Put $\alpha = (\Theta^{-1})^* \lambda_{Dol}$, a Liouville form of finite type on $M_B$ with $d\alpha = \kappa^{-1}\Im\Omega_B$.  The Poisson brackets for $\kappa^{-1}\Omega_B$ are $\kappa$ times those for $\Omega_B$, hence are quadratic by Lemma~\ref{lem:ht-quadratic}.  Corollary~\ref{cor: quadratic poisson}, applied with $\Omega=\kappa^{-1}\Omega_B$, gives $\Phi^*\beta=\alpha+df$.  Since $\beta$ is Weinstein by Lemma~\ref{lem:ht-quadratic}, and $\widehat D$ is weakly Weinstein, Theorem~\ref{thm:W_of_different_Liouville} gives
\[
\Fuk(M_B, \beta) \cong \Fuk(M_B, \alpha) \cong \Fuk(\widehat D, \lambda_{Dol}) ,
\]
the last category being the one computed in \cite{Gammage-McBreen-Webster}.
\end{proof}

\bibliographystyle{plain}
\bibliography{refs.bib}

\end{document}